\documentclass[12pt]{article}

\usepackage{amsthm,amsmath,amssymb}

\usepackage{graphicx}

\usepackage[colorlinks=true,citecolor=black,linkcolor=black,urlcolor=blue]{hyperref}

\theoremstyle{plain}
\newtheorem{theorem}{Theorem}

\newtheorem{corollary}[theorem]{Corollary}
\newtheorem{proposition}[theorem]{Proposition}

\theoremstyle{definition}

\theoremstyle{remark}

\date{}

\title{\bf Narayana polynomials and some generalizations}

\author{Ricky X. F. Chen$^a$, Christian M. Reidys$^b$\\
\small Department of Mathematics and Computer Science,\\[-0.8ex]
\small University of Southern Denmark, Campusvej 55,\\[-0.8ex]
\small DK-5230, Odense M, Denmark\\
\small\tt $^a$chen.ricky1982@gmail.com, $^b$duck@santafe.edu
}

\begin{document}

\maketitle

\begin{abstract}
In this note, by counting some colored plane trees we obtain several binomial identities.
These identities can be viewed as specific evaluations of certain generalizations of the
Narayana polynomials.
As consequences, it provides combinatorial proofs for a bijective problem in Stanley's collection
``Bijective Proof Problems'', a new formula for the Narayana polynomials as well as a new expression for
the Harer-Zagier formula enumerating unicellular maps, in a unified way.
Furthermore, we identify a class of plane trees, whose enumeration is closely
connected to the Schr\"{o}der numbers.
Many other binomial identities are presented as well.


\end{abstract}

\section{Introduction}
This note is firstly motivated by the following bijective problem in Stanley's collection
``Bijective Proof Problems'' \cite[(15)]{12}:
\begin{align}\label{1e4}
\sum_{k=0}^n{n\choose k}^2x^k=\sum_{j=0}^n{n\choose j}{2n-j\choose
n}(x-1)^j.
\end{align}
We note that in Wilf \cite[p.117]{13}, there is ``an odd kind of a
combinatorial proof" of \eqref{1e4} based on the Sieve Method in view of
generatingfunctionology and by counting $n$-subsets of $[2n]$ subject to
some $n$ properties.
Secondly, the authors found that the new expression of the Narayana polynomials
obtained in~\cite{2} (and independently in~\cite{10}) is actually
related to \eqref{1e4} and the new expression of the Narayana polynomials reads
\begin{align}\label{1e5}
\sum_{k=1}^n N_{n,k}y^k=\sum_{k=0}^n \frac{1}{n+1}{n+1\choose k}{2n-k\choose n}(y-1)^k,
\end{align}
where the Narayana numbers $N_{n,k}=\frac{1}{n}{n\choose k}{n\choose k-1}$.
\par
An outline of this note is as follows. By counting some kind of colored plane trees, we
obtain an elementary identity in Section~$2$.
This identity implies the
identities \eqref{1e4} and \eqref{1e5} as special cases.
Furthermore, we present a new
expression for the Harer-Zagier formula, i.e., the generating polynomial for unicellular
maps \cite{17,15}.
Extracting the corresponding coefficients, we show combinatorially that the latter are
equal to the Lehman-Walsh formula \cite{17,16}, which counts unicellular maps for fixed number of
edges and topological genus. We also present a (possibly new) class of plane trees the
enumeration of which involves the Schr\"{o}der numbers.
In Section~$3$, we enumerate some variations of the colored plane trees so that more
binomial identities are obtained.

\section{Narayana polynomials and certain generalization}

In this section, by studying certain kind of
plane trees, we will combinatorially prove an identity which implies \eqref{1e4} and \eqref{1e5}
as special cases.
A colored-labeled plane tree with $n+1$ vertices is a plane tree with vertices
uniquely labeled by $[n+1]=\{1,2,\ldots,n+1\}$ and where its leaves are either not
colored, or colored $N$ or $Y$.
In the following, we denote the sets of the internal vertices, $Y$-leaves and $N$-leaves
in a colored-labeled plane tree $T$ by $int(T)$, $lev_Y(T)$ and $lev_N(T)$, respectively.
As usual, the cardinality of a set $S$ is denoted by $|S|$.
\par
With foresight, we next recall a bijection between labeled plane trees and sets of matches,
called \emph{Chen's bijective algorithm} \cite{1}.

\begin{theorem}[Chen \cite{1}]\label{3t1}
There is a bijection between labeled plane trees with labels in $[n+1]$ and sets of n
matches with labels in $\{1,\ldots, n+1,(n+2)^*,\ldots,(2n)^*\}$,
where a match is a labeled plane tree with two vertices. In addition,
vertices with labels in $[n+1]$ which appear as roots in a set of n matches
appear as internal vertices in its corresponding labeled plane tree,
while vertices with labels in $[n+1]$ which are leaves in a set of $n$ matches
appear as leaves in the corresponding labeled plane tree.
\end{theorem}

For details with respect to \emph{Chen's bijective algorithm}, we refer the reader
to \cite{1}.

We also remark that all proofs in the present note can be formulated directly in
the language of matches.

Now, we are ready to present our results.
Let $\Gamma_{n,x,q}$ denote the set of colored-labeled plane trees $T$ on $[n+x+q]$,
where all vertices with labels in $[q]$ are all uncolored leaves, all vertices with labels in
$\{q+1,\cdots, q+x\}$ are internal.

\begin{proposition}\label{3p1}
The number of trees $T\in \Gamma_{n,x,q}$ such that $|int(T)|+|lev_Y(T)|=k+x$ is given by
$$
{n\choose k}{k+n+x+q-2\choose n+q-1}(n+x+q-1)!,
$$
where $n,x,q,k\geq 0$.
\end{proposition}
\begin{proof}
Based on Theorem \ref{3t1}, it is not difficult to see that the set of colored-labeled plane
trees $T\in \Gamma_{n,x,q}$ with $|int(T)|+|lev_Y(T)|=k+x$ are in
bijection with the set of pairs $(A,\chi)$ where $A\subseteq
[n+x+q]\setminus [x+q]$ with $|A|=k$ and $\chi$ is a set of matches
with labels in $\{1,\ldots,n+x+q,(n+x+q+1)^*,\ldots,2(n+x+q-1)^*\}$
where all vertices with labels in $\{q+1,\cdots,q+x\}$ are roots and
other unstarred roots of $\chi$ are in $A$. (Figure $1$ shows an
example of this bijection. Note, leaves with labels in $A$ will
be colored $Y$ and others $N$.) However, there are ${n\choose k} $ ways
to choose $A$, ${k+n+x+q-2\choose n+q-1}$ ways to determine the
remaining $n+q-1$ roots of $\chi$ besides those $x$ prescribed roots
and at last $(n+x+q-1)!$ ways to match up, whence Proposition \ref{3p1}.
\end{proof}
\begin{center}
\setlength{\unitlength}{1mm}
\begin{picture}(120,40)
\put(15,35){\circle*{1.5}}\put(15,35){\line(-4,-5){8}}
\put(15,35){\line(0,-1){10}}\put(15,35){\line(4,-5){8}}
\multiput(7,25)(8,0){3}{\circle*{1.5}}
\put(7,25){\line(-1,-2){5}} \put(7,25){\line(2,-5){4}}
\put(2,15){\circle*{1.5}}\put(11,15){\circle*{1.5}}
\put(15,25){\line(0,-1){10}} \put(15,15){\circle*{1.5}}
\put(15,15){\line(-1,-3){4}}\put(15,15){\line(1,-3){4}}
\put(11,3){\circle*{1.5}}\put(19,3){\circle*{1.5}}
\put(15,25){\line(4,-5){8}}\put(23,15){\circle*{1.5}}
\put(23,15){\line(0,-1){12}}\put(23,3){\circle*{1.5}}
\put(16,35){\text{3}}\put(4,26){\text{10}}\put(0,12){\text{2}}\put(10,11){\text{6}}
\put(7,14){\text{N}}\put(16,26){\text{8}}
\put(16,14){\text{11}}\put(7,2){\text{5}}\put(12,-1){\text{Y}}
\put(17,-1){\text{1}}\put(24,25){\text{Y}}\put(23,21){\text{7}}
\put(25,15){\text{4}}\put(24,-1){\text{9}}\put(25,3){\text{N}} 
\put(30,26){\vector(1,0){8}}\put(38,25){\vector(-1,0){8}}
\multiput(50,17)(7,0){10}{\circle*{1}}\multiput(50,4)(7,0){10}{\circle*{1}}
\multiput(50,17)(7,0){10}{\line(0,-1){13}}
\put(49,19){\text{4}}\put(49,0){\text{9}}\put(55,19){\text{10}}\put(56,0){\text{2}}
\put(62,19){\text{$13^*$}}\put(63,0){\text{6}}\put(70,19){\text{3}}
\put(69,0){\text{$14^*$}}\put(76,19){\text{$11$}}\put(77,0){\text{5}}
\put(83,19){\text{$16^*$}}\put(84,0){\text{$1$}}\put(91,19){\text{$8$}}
\put(90,0){\text{$17^*$}}\put(97,19){\text{$18^*$}}\put(97,0){\text{$12^*$}}
\put(104,19){\text{$15^*$}}\put(105,0){\text{$19^*$}}
\put(113,19){\text{$20^*$}}\put(113,0){\text{7}} 
\put(42,28){\text{$A=\{5,7,8,10,11\}$}}\put(42,20){\text{$\chi:$}}
\put(40,24){\text{$\Big\{$}}
\end{picture}
\end{center}
\begin{center}
Figure 1: A tree in $\Gamma_{11,2,2}$ and its
corresponding pair $(A,\chi)$.
\end{center}

As an application of Proposition~\ref{3p1} we immediately obtain
a combinatorial proof of the following, well-known identity:

\begin{theorem} For all $n\geq 0, r\in\mathbb{C},q\in \mathbb{Z}$, there holds
\begin{align}\label{3e1}
\sum_{k= 0}^n(-1)^{n-k}{n\choose k}{k+r\choose n+q}={r\choose q}.
\end{align}
\end{theorem}

\begin{proof}
If we weigh each tree $T$ in $\Gamma_{n,x,q}$ by $(-1)^{|lev_N(T)|}$, we observe that
the number $\sum_{k=0}^n(-1)^{n-k}{n\choose
k}{k+n+x+q-2\choose n+q-1}(n+x+q-1)!$ counts the total weight over
all trees in $\Gamma_{n,x,q}$ from Proposition
\ref{3p1}.
However, there is a
very simple involution on $\Gamma_{n,x,q}$: we change $Y$ into $N$ if the
colored leaf with the smallest label is colored $Y$ and vice
versa. From this simple involution, all weights over $\Gamma_{n,x,q}$
cancel out except for those trees, whose leaves are all vertices with labels in
$[q]$, i.e., no colored leaves. Therefore, the total weight here also counts the
number of trees in $\Gamma_{n,x,q}$ in which the leaves
are all vertices in $[q]$. We can obtain this number by enumerating
the corresponding sets of matches according to Theorem \ref{3t1} as
follows: since all leaves of those trees are in $[q]$, we only need
to select out $n+x-1$ vertices from
$\{(n+x+q+1)^*,\ldots,2(n+x+q-1)^*\}$ to be leaves of the matches and
match up with the roots of the matches. Therefore, the number of trees in
$\Gamma_{n,x,q}$, whose leaves are all vertices in $[q]$ is
$$
{n+x+q-2\choose n+x-1}(n+x+q-1)!.
$$
Hence, for $n,x,q\geq 0$,
\begin{align}\label{3e2}
\sum_{k= 0}^n(-1)^{n-k}{n\choose k}{k+n+x+q-2\choose
n+q-1}={n+x+q-2\choose q-1}.
\end{align}
Next we look into~\eqref{3e2} in detail: firstly for any fixed
$q\geq 1$, both sides of \eqref{3e2} are polynomials in $x$. Since
it holds for all $x\geq 0$, it holds for any $x\in \mathbb{C}$.
Secondly, for any $1-n\leq q<1$, the right side of \eqref{3e2}
equals $0$ because $q-1<0$. However, the term ${k+n+x+q-2\choose
n+q-1}$ on the left side is a polynomial in $k$ with degree $<n$, so
the left side is also $0$ by finite difference argument. At last,
for $q<1-n$, both $q-1<0$ and $n+q-1<0$ hold, which obviously leads
to $0$ for both sides of \eqref{3e2}. Hence, we can state that
\eqref{3e2} holds for all $x\in \mathbb{C}, q\in\mathbb{Z}$.
Setting $n+x+q-2=r$ in \eqref{3e2} will complete the proof.
\end{proof}

The identity in the above theorem is implied by Vandermonde's convolution:
\begin{align}\label{1e3}
\sum_{k=0}^n{r\choose k}{m\choose n-k}={r+m\choose n}.
\end{align}
For certain generalizations of Vandermonde's convolution,
we refer to~\cite{11,3,4} and references therein.

 Similarly, if we weigh each tree $T$ in $\Gamma_{n,x,q}$ by
$z^{|lev_{_N}(T)|}$ instead, then we will obtain the following

\begin{theorem}\label{2t4}
For $n,q\geq 0$, $x\in\mathbb{C}$, we have
\begin{align}\label{3e3}
\sum_{k=0}^n{n\choose k}{2n+x+q-k-2\choose
n+q-1}z^k=
\sum_{k=0}^n{n\choose k}{n+x+q-2\choose k+q-1}(z+1)^k.
\end{align}
\end{theorem}
\begin{proof}
Obviously, from Proposition \ref{3p1}, that the total weight over all
trees in $\Gamma_{n,x,q}$ is
$$
\sum_{k=0}^nz^{n-k}{n\choose k}{k+n+x+q-2\choose n+q-1}(n+x+q-1)!.
$$
Counting in another approach, we firstly have the total weight over all trees
$T\in \Gamma_{n,x,q}$ with $|int(T)|=x+k$
$$
{n\choose k}{n+x+q-2\choose n+q-1-k}(n+x+q-1)!(z+1)^{n-k}.
$$
It follows from Theorem~\ref{3t1} that there are ${n\choose k}$ ways to
choose the roots of the matches (internal vertices in the corresponding
tree) besides the $x$ prescribed ones. Furthermore, there are ${n+x+q-2\choose
n+q-1-k}$ ways to choose the remaining roots from starred vertices and
arrange all the roots in $(n+x+q-1)!$ different ways. Finally, all
$n-k$ leaves among all leaves except for those in $[q]$ can be either
colored $Y$ or $N$, whence each of them contributes a weight of $(z+1)$.
Summing over all $0\leq k\leq n$, we also obtain the
total weight over all trees in $\Gamma_{n,x,q}$
$$
\sum_{k=0}^n{n\choose k}{n+x+q-2\choose
n+q-1-k}(n+x+q-1)!(z+1)^{n-k}.
$$
Hence,
$$
\sum_{k=0}^nz^{n-k}{n\choose k}{k+n+x+q-2\choose
n+q-1}=\sum_{k=0}^n{n\choose k}{n+x+q-2\choose n+q-1-k}(z+1)^{n-k},
$$
completing the proof.
\end{proof}

Although eq.~\eqref{2t4} appears quite fundamental, it leads to many results people interested.

\begin{corollary}[Stanley~\cite{12}]
For $n\geq 0$, we have
\begin{align}
\sum_{k=0}^n{n\choose k}^2z^k=\sum_{j=0}^n{n\choose j}{2n-j\choose
n}(z-1)^j.
\end{align}
\end{corollary}
\begin{proof}
Taking $x=q=1, z=z-1$ in \eqref{3e3}, we obtain the bijective proof problem
of Stanley \cite[(15)]{12} as stated in the corollary.
\end{proof}
The following corollary gives the new expression of the Narayana polynomials obtained in~\cite{2}
and in~\cite{10}.
\begin{corollary}
For $n\geq 0$, the Narayana numbers $N_{n,k}$ satisfy
\begin{align}
\sum_{k=1}^n N_{n,k}z^k=\sum_{k=0}^n \frac{1}{n+1}{n+1\choose k}{2n-k\choose n}(z-1)^k.
\end{align}
\end{corollary}
\begin{proof}
Setting $q=0,x=2, z=z-1$ in~\eqref{3e3} completes the proof.
\end{proof}

It is well-known that the large Schr\"{o}der number $S_n$ \cite{18,19}, which counts the number of
plane trees having $n$ edges with leaves colored by one of two colors (say color $Y$
and color $N$), equals the evaluation at $z=2$ in the $n$-th Narayana polynomial, i.e.,
$$
S_n=\sum_{k=1}^n N_{n,k}2^k.
$$
For the case $q=0,x=2$, i.e., $\Gamma_{n,2,0}$, Theorem~\ref{2t4} implies:

\begin{theorem}\label{2t7}
Denote $T_{n+1}$ the number of plane trees of $n+1$ edges with $2$ different internal, marked
vertices and bi-colored leaves.
Then, $T_{n+1}$ is equal to the number of plane trees having $n$ edges with $2$ different, marked
vertices and bi-colored leaves, i.e., $T_{n+1}={n+1\choose 2}S_n$.
\end{theorem}

\begin{proof}
From the proof of Theorem~\ref{2t4}, we know that the number of trees in $\Gamma_{n,2,0}$ is
\begin{align*}
\sum_{k=0}^n{n\choose k}{n\choose
k+1}(n+1)!2^{k}.
\end{align*}
Then, ``deleting" the labels we obtain the number of (unlabelled) plane trees of $n+1$ edges
with $2$ different internal, marked vertices and bi-colored leaves to be
\begin{align*}
\frac{1}{2!n!}\sum_{k=0}^n{n\choose k}{n\choose
k+1}(n+1)!2^{k}=\frac{(n+1)n}{2}\sum_{k=1}^n N_{n,k}2^k={n+1\choose 2}S_n,
\end{align*}
which completes the proof.
\end{proof}

Since the large (and small) Schr\"{o}der numbers satisfy the recurrence \cite{18}
\begin{align}
3(2n-1)S_n=(n+1)S_{n+1}+(n-2)S_{n-1}, n\geq 2
\end{align}
and $S_1=S_2=1$, we have the following corollary.

\begin{corollary}\label{2c8}
The numbers $T_n$ satisfy
\begin{align}
3(2n-1)(n-1)(n+2)T_{n+1}=n(n-1)(n+1)T_{n+2}+(n-2)(n+1)(n+2)T_n.
\end{align}
\end{corollary}

It would be interesting to give direct bijections between the two kinds of trees
defined in Theorem \ref{2t7} and the recurrence in Corollary \ref{2c8}.
\par

Based on Theorem \ref{3e3}, we can also give a new expression for the generating polynomials of
unicellular maps.
A unicellular map is a triangulation, that is a cell-complex of a closed orientable surface.
The number of handles of this surface is called the genus of the map. Denote $A(n,g)$ the number
of unicellular maps of genus $g$ with $n$ edges. The Harer-Zagier formula \cite{15} gives a generating
polynomial for these numbers, which reads
\begin{align}\label{2e3}
\sum_{g\geq 0} A(n,g)x^{n+1-2g}=\frac{(2n)!}{2^n n!}\sum_{k\geq 1}2^{k-1}{n\choose k-1}{x\choose k}.
\end{align}
Now we can give a new expression:
\begin{corollary}
\begin{align}\label{2e4}
\sum_{g\geq 0} A(n,g)x^{n+1-2g}=\frac{(2n)!}{2^n n!}\sum_{k\geq 0}{n\choose k}{x+n-k\choose n+1}.
\end{align}
\end{corollary}
\begin{proof}
Setting $q=2, x=x-n, z=1$ in \eqref{3e3}, we obtain
\begin{align*}
\sum_{k\geq 0}{n\choose k}{x+n-k\choose n+1}=\sum_{k\geq 1}{n\choose k-1}{x\choose k}2^{k-1},
\end{align*}
whence the corollary.
\end{proof}

In fact, there is also an explicit formula for $A(n,g)$ given firstly by Walsh and Lehman \cite{16}. In
the Lehman-Walsh formula, $A(n,g)$ is connected to the number of certain permutations with only odd
cycles \cite{17,20}.
Specifically, let $O(n+1,g)$ denote the number of permutations on $[n+1]$ which consists
of $n+1-2g$ odd cycles, then the Lehman-Walsh formula can be expressed as
\begin{align}
A(n,g)=\frac{(2n)!}{(n+1)!n!2^{2g}}O(n+1,g).
\end{align}
On the other hand, $A(n,g)$ should be equal to the coefficient of the term $x^{n+1-2g}$ on the
right side of the eq.~\eqref{2e3}. To the best of our knowledge, there is no simple combinatorial argument to
show that the coefficient is equal to the Lehman-Walsh formula. However, in the following,
we will show how to obtain the Lehman-Walsh formula from the new expression in above corollary.
\par
Firstly, $A(n,g)$ is equal to the coefficient of the term $x^{n+1-2g}$ on the right side of the
eq.~\eqref{2e4}. We further set $A(n,g)=\frac{(2n)!}{2^n n!(n+1)!}\bar{A}(n,g)$ and denote
the coefficient of the term $x^m$ in the function $f(x)$ as $[x^m]f(x)$.
Then, we have
\begin{align}\label{2e5}
&\bar{A}(n,g)\nonumber\\
&=\sum_{k=0}^n {n\choose k}[x^{n+1-2g}](x+n-k)(x+n-k-1)\cdots x(x-1)\cdots (x-k)\nonumber\\
&=\sum_{k=0}^n {n\choose k}[x^{n+2-2g}][(x+n-k)(x+n-k-1)\cdots x][x(x-1)\cdots (x-k)]
\end{align}
Note there holds
\begin{align}
x(x+1)(x+2)\cdots (x+n-1)&=\sum_k C(n,k)x^k,\\
x(x-1)(x-2)\cdots (x-n+1)&=\sum_k (-1)^{n-k}C(n,k)x^k,
\end{align}
where $C(n,k)$ is the unsigned Stirling number of the first kind, i.e., $C(n,k)$ counts the
number of permutations on $[n]$ with $k$ cycles. Therefore, from eq.~\eqref{2e5} we have
\begin{align}
\bar{A}(n,g)=\sum_{k=0}^n {n\choose k}\sum_{i+j=n+2-2g}C(n-k+1,i)(-1)^{k+1-j}C(k+1,j).
\end{align}
Inspecting this expression, it is not clear why the right hand side should be always a positive number.
However, we shall prove

\begin{theorem}
For $n,g\geq 0$, we have
\begin{align}\label{2e6}
\sum_{k=0}^n {n\choose k}\sum_{i+j=n+2-2g}C(n-k+1,i)(-1)^{k+1-j}C(k+1,j)=2^{n-2g}O(n+1,g).
\end{align}
\end{theorem}

\begin{proof}
We firstly construct a set of objects which is counted by the left hand side of \eqref{2e6}.
Let $[n+1]^*=\{0,1,2,\ldots,n+1\}$ and consider the set $\mathcal{T}$ of all pairs
$(\alpha,\beta)$ where $\alpha$ is a permutation on $A\subset [n+1]^*$ with $0\in A$ while
$\beta$ is a permutation on $[n+1]^*\setminus A$ with $n+1\in[n+1]^*\setminus A$,
and where the total number of cycles in $\alpha$ and $\beta$ is $n+2-2g$.
Denote the difference between the number of elements in $\beta$ and the number of cycles
in $\beta$ as $d(\beta)$, and weigh each pair by $W[(\alpha,\beta)]=(-1)^{d(\beta)}$.
Then, the total weight over all pairs in $\mathcal{T}$ is counted by the left side of the
eq.~\eqref{2e6}, i.e.,
\begin{align}
\sum_{(\alpha,\beta)\in \mathcal{T}}W[(\alpha,\beta)]=
\sum_{k=0}^n {n\choose k}\sum_{i+j=n+2-2g}C(n-k+1,i)(-1)^{k+1-j}C(k+1,j).
\end{align}
Next, we will prove
\begin{align}
\sum_{(\alpha,\beta)\in \mathcal{T}}W[(\alpha,\beta)]=2^{n-2g}O(n+1,g).
\end{align}
Denote the length of the cycle containing $n+1$ in $\beta$ as $|\beta_{n+1}|$ while
the length of the cycle containing $0$ in $\alpha$ as $|\alpha_0|$.
\par
There is a bijection, $\phi$, between the set $X$ of pairs $(\alpha,\beta)$ with
$|\beta_{n+1}|=2i+1,|\alpha_0|\geq 2$ and the set $Y$ of pairs $(\alpha',\beta')$ with
$|\beta'_{n+1}|=2i+2,|\alpha'_0|\geq 1$, where $0\leq i<g$. $\phi$ is given by
\begin{align}
\phi\colon X \longrightarrow Y,\quad
 (\alpha,\beta)\mapsto (\alpha',\beta'),
\end{align}
where $\alpha'$ is obtained from $\alpha$ by removing the element $\alpha(0)$ while
$\beta'$ is obtained from $\beta$ by inserting $\alpha(0)$ between $n+1$ and $\beta(n+1)$,
i.e., $\beta'(n+1)=\alpha(0), \beta'(\alpha(0))=\beta(n+1)$. Reversing this switch we derive
the reverse map $\phi^{-1}$.
It is obvious that $(\alpha,\beta)$ and $(\alpha',\beta')=\phi((\alpha,\beta))$ carry opposite
signs, whence their weights will cancel. In particular, the total weight over all pairs
$(\alpha,\beta)$ where $|\beta_{n+1}|$ is even, will cancel.
\par
Since there are $n+2-2g$ cycles in $\alpha$ and $\beta$, $|\beta_{n+1}|\leq 2g+1$,
equality will be achieved when all other cycles are singletons (which implies $|\alpha_0|=1$).
Thus, after applying $\phi$, the total weight over all pairs $(\alpha,\beta)$ is
reduced to the total weight over all pairs $(\alpha,\beta)$ where $|\alpha_0|=1$,
$|\beta_{n+1}|=2i+1$ for $0\leq i\leq 2g+1$. We denotes the latter set by $U$.
\par
There is an involution, $\varphi$, over all pairs $(\alpha,\beta)$ in $U$ which
have at least one even cycle.
$\varphi$ is obtained as follows: consider the set of even cycles contained in $\alpha,\beta$.
There is a unique, even cycle containing the smallest element, $c$. If $c$ is contained in $\alpha$
then $\alpha'=\alpha\setminus c$, $\beta'=\beta \cup c$ and $\alpha'=\alpha\cup c$ and
$\beta'=\beta\setminus c$, otherwise.

Note from $\beta$ to $\beta'$, the total number of elements changes by an even number while
the total number of cycles changes by $1$ (odd), thus $\beta$ and $\beta'$ must carry opposite
signs. Therefore, their weight cancels.
\par
Applying $\varphi$, the total weight over all pairs $(\alpha,\beta)$ is further reduced
to the total weight over all pairs $(\alpha,\beta)$ which have total number of $n+2-2g$ odd cycles. It is
clear that all pairs in the latter set ($V$) carry a positive sign, whence the total weight
is equal to the total number of elements in $V$.
\par
Finally, each pair $(\alpha, \beta)\in V$ can be viewed as a partition of all cycles except the
cycle containing $n+1$, of a permutation on $[n+1]$ with $n+1-2g$ odd cycles, into two parts.
Conversely, given a permutation on $[n+1]$ with $n+1-2g$ odd cycles, there are $2^{n-2g}$
different ways to partition all cycles except the one containing $n+1$ into two parts.
Therefore, we have
\begin{align*}
\sum_{(\alpha,\beta)\in V}W[(\alpha,\beta)]=|V|=2^{n-2g}O(n+1,g).
\end{align*}
Hence,
\begin{align*}
\sum_{(\alpha,\beta)\in \mathcal{T}}W[(\alpha,\beta)]=2^{n-2g}O(n+1,g),
\end{align*}
completing the proof.
\end{proof}

Accordingly, we obtain the Lehman-Walsh formula
\begin{align*}
A(n,g)=\frac{(2n)!}{(n+1)!n!2^n}\bar{A}(n,g)=\frac{(2n)!}{(n+1)!n!2^{2g}}O(n+1,g).
\end{align*}

\section{Partial sums and special cases}

In this section, we consider further generalizations based on the discussion in Section $2$.

\begin{theorem} For $0\leq q$, $0\leq n_1\leq n$, $x\in\mathbb{C}$, we have
\begin{multline}\label{3e4}
\sum_{k=0}^{n_1}{n\choose k}{n+x+q+k-2\choose
n+q-1}z^{n-k}=\\
\sum_{k=0}^{n_1}{n\choose k}{n+x+q-2\choose q+n-1-k}\sum_{i=0}^{n_1-k}{n-k\choose i}z^{n-k-i}.
\end{multline}
\end{theorem}
\begin{proof}
Firstly it is obvious from Proposition \ref{3p1} that the total weight over all
trees $T$ in $\Gamma_{n,x,q}$ with $|int(T)|+|lev_Y(T)|\leq x+n_1$ is
$$
\sum_{k=0}^{n_1}z^{n-k}{n\choose k}{k+n+x+q-2\choose n+q-1}(n+x+q-1)!.
$$
However, the total weight over all trees $T$ in $\Gamma_{n,x,q}$ with $|int(T)|=x+k$
($k\leq n_1$) is
$$
{n\choose k}{n+x+q-2\choose n+q-1-k}(n+x+q-1)!\sum_{i=0}^{n_1-k}{n-k\choose i}z^{n-k-i},
$$
since we can color at most $n_1-k$ leaves as $Y$ if there are $k$ roots (of the matches)
from $\{x+q+1,\ldots,x+q+n\}$.
Summing over all $0\leq k\leq n_1$, we also obtain the
total weight over all trees $T$ in $\Gamma_{n,x,q}$ with $|int(T)|+|lev_Y(T)|\leq x+n_1$
$$
\sum_{k=0}^{n_1}{n\choose k}{n+x+q-2\choose
n+q-1-k}(n+x+q-1)!\sum_{i=0}^{n_1-k}{n-k\choose i}z^{n-k-i}.
$$
Hence,
\begin{align*}
\sum_{k=0}^{n_1}z^{n-k}{n\choose k}{k+n+x+q-2\choose
n+q-1}
=\sum_{k=0}^{n_1}{n\choose k}{n+x+q-2\choose n+q-1-k}\sum_{i=0}^{n_1-k}{n-k\choose i}z^{n-k-i}.
\end{align*}
\end{proof}

\begin{corollary} For $0\leq q$, $0\leq n_1\leq n$, $x\in\mathbb{C}$, we have
\begin{multline}\label{3e5}
\sum_{k=0}^{n_1}(-1)^{n-k}{n\choose k}{k+x+q+n-2\choose q+n-1}\\
=\sum_{k=0}^{n_1}(-1)^{n-n_1}{n\choose k}{x+q+n-2\choose q+n-1-k}{n-k-1\choose n_1-k}.
\end{multline}
\end{corollary}

\begin{proof}
Note the left hand side is the total weight over all trees $T$ where $|int(T)|+|lev_N(T)|\leq n_1$ and each
tree $T$ is weighed by $(-1)^{|lev_N(T)|}$. (Note both sides have a factor $(n+x+q-1)!$, they will
be canceled out.) Similar as in the proof of Theorem $2.3$, we apply the involution.
In this case, all weight over trees $T$ where $|int(T)|+|lev_N(T)|<x+n_1$ or $|int(T)|+|lev_N(T)|=x+n_1$
and the colored leave with the smallest label is colored $Y$ will cancel. Thus, the total weight is equal to
the weight over trees $T$, where $|int(T)|+|lev_N(T)|=x+n_1$ and the colored leave with the smallest
label is colored $N$. This number is calculated as follows:
\begin{itemize}
\item[1.] select $k$ labels from $\{x+q+1,\ldots,x+q+n\}$ as the roots of the matches in ${n\choose k}$
different ways;
\item[2.] choose from the starred labels $n+q-1-k$ labels as other roots of the matches in
${n+x+q-2\choose n+q-1-k}$ ways;
\item[3.] color the leave with the smallest label $N$, and choose another $n_1-k$ leaves with labels in
  $\{x+q+1,\ldots,x+q+n\}$ in ${n-k-1\choose n_1-k}$ different ways and color all of them $Y$
  (the rest will be colored $N$).
\end{itemize}
Note that each such tree has weight $(-1)^{n-n_1}$. Summing over $0\leq k\leq n_1$,
we obtain for the total weight
  \begin{align*}
  \sum_{k=0}^{n_1}(-1)^{n-n_1}{n\choose k}{n+x+q-2\choose n+q-1-k}{n-k-1\choose n_1-k},
  \end{align*}
  which equals the right hand side of \eqref{3e5}, whence the corollary.
  \end{proof}

We remark that it is interesting to observe from the right hand side of \eqref{3e5} that
the partial sum on the left hand side alternates in sign as $n_1$ goes from $0$ to $n$.
Setting $x=q=1$ and $q=0,x=2$ respectively, we obtain

\begin{corollary}For $0\leq n_1\leq n$, we have
\begin{align}
\sum_{k=0}^{n_1}(-1)^{n_1+k}{n\choose k}{k+n\choose n}&=\sum_{k=0}^{n_1}{n\choose k}^2{n-k-1\choose n_1-k},\\
\sum_{k=0}^{n_1}N_{n,k+1}{n-k-1\choose n_1-k}&=\sum_{k=0}^{n_1}\frac{(-1)^{k+n_1}}{k+n+1}{n\choose k}{k+n+1\choose n}.
\end{align}
\end{corollary}

\begin{corollary}The following partial sum holds:
\begin{align}
\sum_{k=0}^n(-1)^k{x\choose k}=(-1)^n{x-1\choose n}.
\end{align}
\end{corollary}
\begin{proof}
Comparing the right hand side of \eqref{3e4} when $z=-1$ and the right hand side of \eqref{3e5},
and noting that ${n\choose k}$ form a basis, we have
$$
\sum_{i=0}^{n_1-k}(-1)^i{n-k\choose i}=(-1)^{n_1-k}{n-k-1\choose n_1-k}.
$$
Since this holds for every $n-k$ as a polynomial of $n-k$, it holds for any $x$.
\end{proof}

\begin{theorem}For $0\leq n, 1\leq k, 0\leq t<k, q\in \mathbb{Z},x\in \mathbb{C},\omega_k=e^{\frac{j2\pi}{k}}, j^2=-1$,
we have
\begin{multline}\label{3e36}
\sum_{l=0}^n{kn+t\choose kl+t}{kl+x+q+kn+2t-2\choose q+kn+t-1}z^{kl}=\\
\frac{1}{k}\sum_{i=0}^{kn+t}{kn+t\choose i}{x+q+kn+t-2\choose q+kn+t-1-i}
                     z^{i-t}\sum_{l=1}^k \frac{(1+z\omega_k^l)^{kn+t-i}}{(\omega_k^l)^{t-i}}.
\end{multline}
\end{theorem}
\begin{proof}
Considering the trees $T\in \Gamma_{kn,x,q}$ with $|int(T)|+|lev_Y(T)|=x+kl+t, l\geq 0$,
we obtain, along the lines of Theorem \ref{2t4}
\begin{multline*}
\sum_{l=0}^n {kn+t\choose kl+t}{kl+kn+x+q+2t-2\choose kn+t+q-1}z^{kn-kl}=\\
\sum_{i=0}^{kn+t}{kn+t\choose i}{kn+t+x+q-2\choose kn+t+q-i-1}\sum_{kl+t\geq i}{kn+t-i\choose kl+t-i}z^{kn-kl}.
\end{multline*}
Canceling out the term $z^{kn}$ and setting $z=z^{-1}$ in the above identity, the
second summation on the right hand side becomes
\begin{align*}
z^{i-t}\sum_{l\geq 0}{kn+t-i\choose kl+t-i}z^{kl+t-i}.
\end{align*}
From the identity~\cite[eq.~$(1.53)$]{14}, we have
\begin{align*}
\sum_{i\geq 0}{n\choose a+ki}x^{a+ki}=\frac{1}{k}\sum_{i=1}^k (\omega_k^i)^{-a}(1+x\omega_k^i)^{n}, k-1\geq a, a\in \mathbb{Z}.
\end{align*}
Therefore, setting $a=t-i,n=kn+t-i,k=k$ we obtain
\begin{align*}
\sum_{l\geq 0}{kn+t-i\choose kl+t-i}z^{kl+t-i}=\frac{1}{k}\sum_{l=1}^k(\omega_k^l)^{i-t}(1+z\omega_k^l)^{kn+t-i}
\end{align*}
and the proof follows.
\end{proof}

Finally, we summarize some particular evaluations of \eqref{3e36} in the following three
corollaries.

\begin{corollary}For $0\leq n, 1\leq k, 0\leq t<2, 0\leq s<4$, we have
\begin{align}
\sum_{l=0}^n{n\choose l}{n+x+q+l-2\choose n+q-1}z^{l}=
\sum_{l=0}^n{n\choose l}{x+q+n-2\choose q+n-1-l}z^l(1+z)^{n-l} \hfill
\end{align}\vspace{-0.5cm}
\begin{multline}
\sum_{l=0}^{n}2{2n+t\choose 2l+t}{2l+x+q+2n+2t-2\choose q+2n+t-1}z^{2l}=\\
\sum_{l=0}^{2n+t}{2n+t\choose l}{x+q+2n+t-2\choose q+2n+t-1-l}\frac{(z-1)^{2n+t-l}+(z+1)^{2n+t-l}}{z^{t-l}}
\end{multline}\vspace{-0.5cm}
\begin{multline}
\sum_{l=0}^{n}4{4n+s\choose 4l+s}{4l+x+q+4n+2s-2\choose q+4n+s-1}z^{4l}=\sum_{l=0}^{4n+s}{4n+s\choose l}{x+q+4n+s-2\choose q+4n+s-1-l}\cdot\\
 \frac{(z-1)^{4n+s-l}+(z+1)^{4n+s-l}+(z+j)^{4n+s-l}+(z-j)^{4n+s-l}}{z^{s-l}}
\end{multline}
\end{corollary}

\begin{proof}
Setting $k=1,2,4$ in \eqref{3e36}, we obtain $(28)-(30)$, respectively.
\end{proof}

\begin{corollary}For $0\leq n, 1\leq k, 0\leq t<k, 0\leq s_1<2, 0\leq s_2<4, \omega_k=e^{\frac{j2\pi}{k}}, j^2=-1$, there holds
\begin{align}
\sum_{i=0}^{kn+t}{kn+t\choose i}^2\sum_{l=1}^k\frac{(1+\omega_k^l)^{kn+t-i}}{(\omega_k^l)^{t-i}}&=\sum_{l=0}^n k{kn+t\choose kl+t}{kn+kl+2t\choose kn+t}\\
\sum_{i=0}^{kn+t}N_{kn+t,i+1}\sum_{l=1}^k\frac{(1+\omega_k^l)^{kn+t-i}}{(\omega_k^l)^{t-i}}&=\sum_{l=0}^n \frac{1}{n}{kn+t\choose kl+t}{kn+kl+2t\choose kl+t+1}\\
\sum_{l=0}^n2{2n+s_1\choose 2l+s_1}{2l+2n+2s_1\choose 2n+s_1}-1&=\sum_{l=0}^{2n+s_1}{2n+s_1\choose l}^22^{2n+s_1-l}
\end{align}\vspace{-0.5cm}
\begin{align}
\sum_{l=0}^{2n+s_1}N_{2n+s_1,l+1}2^{2n+s_1-l}&=\sum_{l=0}^n \frac{1}{n}{2n+s_1\choose 2l+s_1}{2n+2l+2s_1\choose 2l+s_1+1}\\
\sum_{l=0}^{4n+s_2}N_{4n+s_2,l+1} \left[2^{4n+s_2-l}+\frac{(1+j^{l-s_2})}{(1+j)^{l-s_2-4n}} \right]&=\sum_{l=0}^n \frac{1}{n}{4n+s_2\choose 4l+s_2}{4n+4l+2s_2\choose 4l+s_2+1}\\
\sum_{l=0}^{4n+s_2}{4n+s_2\choose l}^2 \left[\frac{2^{4n}}{2^{l-s_2}}+\frac{(1+j^{l-s_2})}{(1+j)^{l-s_2-4n}} \right] &=\sum_{l=0}^n4{4n+s_2\choose 4l+s_2}{4l+4n+2s_2\choose 4n+s_2}-1
\end{align}\vspace{-0.5cm}

\end{corollary}
\begin{proof}
Setting $x=q=1,z=1$ in \eqref{3e36} leads to eq.~$(31)$;
Setting $q=0, x=2, z=1$ in \eqref{3e36} leads to eq.~$(32)$.
Setting $k=2,4$ in eq.~$(31)$ leads to eq.~$(33)$ and respectively eq.~$(36)$;
Setting $k=2,4$ in eq.~$(32)$ leads to eq.~$(34)$ and respectively eq.~$(35)$,
completing the proof.
\end{proof}
\begin{corollary}For $n\geq 0$, we have
\begin{align}
\sum_{l=0}^n{2n\choose 2l}{2l+2n\choose 2n}&=\sum_{l=0}^{n-1}{2n\choose 2l+1}{2n+2l+1\choose 2n}+1\\
\sum_{l=0}^n{2n+1\choose 2l}{2l+2n+1\choose 2n+1}&=\sum_{l=0}^{n}{2n+1\choose 2l+1}{2n+2l+2\choose 2n+1}-1
\end{align}
\end{corollary}
\begin{proof}
Using both eq.~$(7)$ and eq.~$(33)$ completes the proof.
\end{proof}
\subsection*{Acknowledgements}

We acknowledge the financial support of the Future and Emerging
Technologies (FET) programme within the Seventh Framework Programme (FP7) for
Research of the European Commission, under the FET-Proactive grant agreement
TOPDRIM, number FP7-ICT-318121.


\end{document}